\documentclass[a4]{article}

\usepackage{amsmath,amsfonts,amsthm,amssymb,amscd,graphicx}
\theoremstyle{definition}
\newtheorem{thm}{Theorem}[section]
\newtheorem{deff}[thm]{Definition}
\newtheorem{lemm}[thm]{Lemma}
\newtheorem{prop}[thm]{Proposition}
\newtheorem{cor}[thm]{Corollary}
\newtheorem{rem}[thm]{Remark}
\newtheorem{exam}{Example}
\newtheorem{prob}{Problem}

\newcommand{\zahl}{\mathbb{Z}}
\newcommand{\zhat}{\widehat{\zahl}}
\newcommand{\num}{\mathbb{N}}
\newcommand{\real}{\mathbb{R}}
\newcommand{\nhat}{\widehat{\num}}
\newcommand{\zring}{{\bf Z}}
\newcommand{\zn}{\zring_n}
\newcommand{\zm}{{\zring_m}}
\newcommand{\zp}{{\zring_p}}

\newcommand{\tow}{\uparrow}
\newcommand{\ttow}{\mathrel{\rotatebox[origin = c]{90}{$\twoheadrightarrow$}}}
\newcommand{\p}{\lambda}
\newcommand{\lcm}{{\rm lcm}}
\newcommand{\zmn}{\zring_{m,n}}
\newcommand{\od}{\kappa}

\newcommand{\ol}[1]{\overline{#1}}
\newcommand{\fct}[1]{\mathop{\mathrm{#1}}}
\newcommand{\id}{\fct{id}}

\title{Arithmetic Convergence of Double-iterated Polynomials}
\author{Rin Gotou
 \footnote{Department of Mathematics, Osaka University, 1-1 Machikaneyama, Toyonaka, Osaka, 560-0043, Japan} 
}
\date{}

\begin{document}
\maketitle
\begin{abstract}
Let $f$ be a polynomial with integer coefficients such that $f(n)$ positive for any positive integer $n$.
We consider diverging sequences $\{ y_n\}$ given by $y_0 = b$ and $y_{n+1} = f^{y_n}(a)$ with positive integers $a$ and $b$. We show such a sequence converges in $\widehat{\mathbb{Z}}$ and the limit is independent of $b$, if and only if $f$ does not become a permutation of length $p$ on $\mathbb{Z}/p\mathbb{Z}$ for any prime number $p$. We also show that $b'$-adic asymptotic approximations of the equation $f^y(a) = y$ holds in $\mathbb{N}$ for some bases $b'$.
\end{abstract}

\let \thefootnote \relax \footnotetext{ \texttt{u661233h@ecs.osaka-u.ac.jp}}

\section{Introduction}
 In \cite{J-Y}, J. Jim\'{e}nez-Urroz and Yebra proved the following result.
\begin{thm}([J-Y, Theorem 1])
 For any positive integers $a$ and $b$, there exists a positive integer $x$ such that
\begin{equation}
a^x \equiv_{b} x \quad ({\rm i.e.} \quad a^x \equiv x \mod b). \label{axbnx}
\end{equation}
 Moreover, if $b$ is \emph{valid} (i.e. for every pair of prime numbers $p$ and $q$ such that $p\mid b$ and $q\mid p-1$, we have $q \mid b$), then there exists a sequence of positive integers $\{x_n\}$ such that
\begin{equation}
a^{x_n} \equiv_{b^n} x_n \ \ {\rm and} \ \ x_n = c_n b^{n-1} + x_{n-1}\ \ \ (0 \leq c_n < b) \nonumber
\end{equation}
for every $n$. 
\end{thm}
 For example, they gave the case $a=7$ and $b=10$ as the following:
\begin{eqnarray}
7^3 &\equiv_{10}& 3, \nonumber \\
7^{43} &\equiv_{10^2}& 43, \nonumber \\
7^{343} &\equiv_{10^3}& 343, \nonumber \\
&\vdots& \nonumber \\
7^{\cdots 3643331265511565172343} &\equiv_{10^n}& \cdots 3643331265511565172343, \nonumber\\
&\vdots& . \nonumber
\end{eqnarray}
 The proof of Theorem 1.1 was constructive, that is, done by giving an algorithm to obtain $x$. In \cite{S-S}, D. B. Shapiro and S. D. Shapiro proved the former part of Theorem 1.1 independently, affording a more explicit form of the same algorithm. They showed that
\begin{thm}([S-S, Corollary 2.11]) For any positive integers $a$ and $b$, the sequence $a$, $a^a$, $a^{a^a}$, $a \tow \tow 4$, ..., $a \tow \tow n$,... become stable modulo $b$ for $n \gg 0$, where $a \tow \tow n$ is Knuth's up-arrow notation introduced in \cite{K}. Equivalently, the sequence converges in the ring $\zahl_p$ of $p$-adic integers for every prime number $p$.
\end{thm}
 This implies $x = a\tow \tow n$ for $n \gg 0$ satisfies the congruence (\ref{axbnx}). In addition, this allows to state the above example as
\begin{equation}
\lim_{n \to \infty}7 \tow \tow n = \cdots3643331265511565172343 \quad {\rm in} \quad \zahl_2 \times \zahl_5. \nonumber
\end{equation} 

  We regard the above two theorems as results on the polynomial $ax \in \zahl[x]$ from the viewpoint of dynamical systems. Let $X$ be a set and $x$ an element of $X$. For a map $f:X \to X$, we denote $f^n(x)=f \circ f \circ \cdots \circ f(x)$ by $f \tow_x (n)$ and thus obtain a map $f\tow_x :\num \to X$ where $\num$ is the set of the positive integers. If $X=\num$, then for every positive integer $a$, we obtain $f \tow_a:\num \to \num$ and therefore we can construct $f \tow_a \tow_b$. If $f(x) = ax$, then we have $f \tow_1 (n) =a^n = a \tow n$ and $f \tow_1 \tow_1 (n) = (f \tow_1) \tow_1(n)= a \tow \tow n$.

 The purpose of this paper is to generalize Theorems 1.1 and 1.2 to more general polynomials in $\zahl[x]$. We shall prove the following:
\begin{thm}\label{main1}
Let $f$ be a polynomial of one variable with integer coefficients and we assume $f(\num) \subseteq \num$. Then the following conditions are equivalent to each other:
\begin{enumerate}
\item For any prime number $p$, the reduction map $f_p : \zahl/p\zahl \to \zahl/p\zahl$ is not a cyclic permutation of length $p$.
\item For any $a,b \in \num$, if $\lim_{n \to \infty} f \tow_a \tow_b(n) = \infty$ in $\mathbb{R}$, then the limit $\lim_{n \to \infty} f \tow_a \tow_b (n)$ exists in $\zahl_p$ for every prime number $p$ and is independent of $b$.
\end{enumerate}
\end{thm}
 We call a polynomial $f \in \zahl[x]$ \emph{tower-stable} if $f$ satisfies condition (i) in the above theorem (this definition will be generalized in Definition \ref{ts}).
 Note that for every $a \in \num$, $f(x)=ax$ is tower-stable because $f(0)=0$. Therefore Theorem 1.3 implies Theorem 1.2.
\begin{deff}
Let $f$ be a tower-stable polynomial. A positive integer $b$ is said to be $f$-$valid$ if $b$ is square-free, valid and for every pair of prime numbers $p$ and $q$ such that $p \mid b$ and  $q  \mid  \p_f(p)$, it holds $q \mid b$, where $\p_f(p)$ is the period of the reduction map $f_p : \zahl/p\zahl \to \zahl/p\zahl$. 
\end{deff}
 With this refinement of valid numbers, we state the following generalization of the latter part of Theorem 1.1.
\begin{thm}\label{main2}
 Let $f$ be a tower-stable polynomial. If $b$ is $f$-valid, there exists a sequence of positive integers $\{x_n\}$ such that
\begin{equation}
f^{x_n}(a) \equiv_{b^n} x_n \ \ {\rm and}\ \ x_n = c_n b^{n-1} + x_{n-1}\ \ \ (0 \leq c_n < b). \nonumber
\end{equation}
\end{thm}
For example, we choose a polynomial $f(x)=x^2+x+3$ and $b=10$. $f$ is tower-stable because $f(0) = f(-1)=3$ makes no reductions injective. We see $b=10$ is $f$-valid. It is easy to see $10$ is valid: if $p \mid 10$ and $q \mid (p-1)$ then $(p,q) = (5,2)$ and $q \mid 10$. Moreover, by some direct computation, we can see $\p_f(5)=2$ and $\p_f(2)=1$, this leads that $10$ is $f$-valid.
\begin{eqnarray}
f^3(0) &\equiv_{10}& 3, \nonumber \\
f^{43}(0) &\equiv_{10^2}& 43, \nonumber \\
f^{243}(0) &\equiv_{10^3}& 243, \nonumber \\
&\vdots & \nonumber \\	
f^{\cdots 636048243}(0) &\equiv_{10^n}& \cdots 636048243, \nonumber \\
&\vdots &. \nonumber
\end{eqnarray}

The outline of the paper is as follows. In Section 2, We extend maps on $\num$ to some maps on $\zhat$, the profinite complation of $\zahl$. Using these extensions, we discuss dynamical systems on $\zhat$ in Section 3 and prove profinite version Theorem \ref{convzhat} of Theorem \ref{main1}, then obtain Theorem \ref{main1} as a corollary. In Section 4, we give a more precise evaluation of order of convergence to show Theorem \ref{main2}.

\subsection*{Acknowledgement}
Takao Watanabe encouraged me to try to write a paper and gave some valid comments for drafts. Takehiko Yasuda and Seidai Yasuda proofread this paper and gave warm encouragement. Without their contributions, this paper cannot be materialized. I would like to express my greatest appreciation to them.

\subsection{Notation}
 \begin{tabular}{ll}
$\num$ & the set of the all positive integers, $\{1,2,3,...\}$ \\
$[n]$ & the set $\{1,2,3,...,n\}$ \\
$a  \mid  b$ & $a$ divides $b$ \\
$a \equiv_n b$ & $a$ is congruent to $b$ modulo $n$\\
$\zn$ & the ring $\zahl/n\zahl$ \\
$\ol{a}$ & the residue class of $a$\\
$f_n$ & the map $\zm \to \zn$ for some $m$ induced by $f$, mostly $m=n$\\
$f^n$ & the $n$-th iteration of $f$\\
$f \tow_a (n)$ & $f^n(a)$\\
 \end{tabular}

\section{Profinite Completion of $\num$}

 Let $S$ be a finite quotient semigroup of $(\num,+)$ and $\pi:\num \to S$ be the quotient map. The semigroup $S$ is generated by $\pi(1)$. Since $S$ is finite, the set $P_S:=\{(k,l) \in \num^2 \mid \pi(k) = \pi(k+l)\}$ is nonempty. Let $(m,n)=(m_S,n_S)$ be a pair of positive integers such that $n:=\min \{ l \mid$ there exists $k$ such that $(k,l) \in P_S \}$ and $m:= \min \{k \mid (k,n)\in P_S \}$. Then $S$ is uniquely determined by $m$ and $n$ up to isomorphism and identified with the semigroup $\zmn :=([m-1] \cup \zn , \dot{+})$ where the operator $\dot{+}$ is defined as
\[a\, \dot{+} \, b =
	\begin{cases}
		a+b \in [m-1] & ({\rm if} \:\: a,b \in [m-1]\: {\rm and} \:\: a+b < m),\\
		\ol{a+b} \in \zn & ({\rm otherwise})
	\end{cases}
\]
and the quotient map $\pi_{m,n}: \num \to \zmn$ is given by
\[\pi_{m,n}(a) =
	\begin{cases}
		a \in [m-1] & (a <m),\\
		\ol{a} \in \zn & (a \geq m).
	\end{cases}
\]
 If $m \leq k$ and $n \mid l$, there is natural homomorphism $\phi_{(k,l),(m,n)}:\zring_{k,l}\to \zmn$. Thus $(\{ \zmn \},\{ \phi_{(k,l),(m,n)} \} )$ is a projective system and we define $\nhat$, the profinite completion of $\num$ as the projective limit of $(\zmn,\phi_{(k,l),(m,n)})$. The semigroup $\nhat$ has the following description with $\zhat :=\varprojlim_{n} \zn$:
\begin{prop} \label{constnhat}
We have $\nhat \cong (\num \cup \zhat\, ,\widehat{+})$ where
\[a\, \widehat{+} \, b = 
	\begin{cases}
		a+b \in \num & (a,b \in \num)\\
		a+b \in \zhat & ({\rm otherwise}).
	\end{cases}
\]
\end{prop}
\begin{proof} We can decompose every map $\phi_{(k,l),(m,n)}$ as $\phi_{(k,l),(m,n)} = \phi_{(k,n),(m,n)}  \circ \phi_{(k,l),(k,n)}$. Therefore the definition of projective limit allows us to compute $\nhat$ as
\[\varprojlim \zmn = \varprojlim_n \varprojlim_m \zmn. \]
 First, we show $\varprojlim_{m} \zmn \cong (\num \cup \zn , \dot{+})$, where 
\[a\, \dot{+} \, b = 
	\begin{cases}
		a+b \in \num & (a,b \in \num)\\
		\ol{a+b} \in \zn & ({\rm otherwise}).
	\end{cases}
\]
 Indeed we can construct maps which are the inverses of each other: one is given by
\[\varprojlim _m \zmn \ni a = (a_m)_{m\in \num} \mapsto
	\begin{cases}
		a_m \in \num & ({\rm if \ } a_m \in [m-1]\ {\rm for \ any \ }m)\\
		a_1 \in \zn & ({\rm otherwise}).
	\end{cases}
\]
 The other is given by:
\[\num \cup \zn \ni a \mapsto
	\begin{cases}
		(\pi_{m,n}(a))_{m\in \num} & ({\rm if \ } a \in \num)\\
		(a)_{m \in \num} & ({\rm if \ } a \in \zn).
	\end{cases}
\]
 Then it remains to show $\varprojlim_{n} (\num \cup \zn , \dot{+}) = (\num \cup \zhat, \widehat{+})$, which follows from $\varprojlim_{n} \zn = \zhat$.
\end{proof}
 Every projection $\pi_{m,n}:\num \to \zmn$ defines multiplication on $\zmn$. We regard $\zmn$ as a topological semiring with discrete topology. Then, $\nhat$ becomes a compact Hausdorff semiring by the projective limit topology. Note that a sequence $n_1, n_2,...$ of natural numbers converges to $s \in \zhat$ in $\nhat$ if and only if $n_i$ tends to infinity in $\mathbb{R}$ and the sequence $\widehat{n_i}$ converges to $s$ in $\zhat$ (with respect to the standard topology of $\zhat$), where $\widehat{\cdot}:\num \to \zhat$ is the natural embedding.

\begin{prop}\label{profext}
  Let $X$ be a metric space and $f:\num \to X$ a map. Then the following conditions are equivalent:
\begin{enumerate}
\item The map can be extended to a continuous map $\widehat{f}:\nhat \to X$.
\item For every $s \in \zhat$ and every sequence of positive integers $\{n_i \}_{i\in \num}$ that satisfies $n_i \to + \infty$ in $\mathbb{R}$ and $\widehat{n_i} \to s$ in $\zhat$ as $i \to \infty$, the limit $ \lim_{i \to \infty} f(n_i)$ exists in $X$ and is independent of $n_i$ (i.e. the limit $ \lim_{i \to \infty} f(n_i)$ depends only on $s$).
\end{enumerate}
\end{prop}
\begin{proof}(i) $\Rightarrow$ (ii) follows by the above remark about the convergence of sequence of positive integers on $\nhat$; we show (ii) $\Rightarrow$ (i). For every $s \in \zhat$, we take $\{ n_i\}_{i \in \num}$ as in the condition of (ii) and set $\widehat{f}(s) := \lim_{i \to \infty} f(n_i)$. Then $\widehat{f}:\nhat \to X$ is well-defined by assumption, we show $\widehat{f}$ is continuous. Since $\nhat$ is given by a projective limit of a countable system of finite sets, the topology of $\nhat$ is first countable. Thus it is enough to see $\widehat{f}$ is sequentially continuous. Let $\{s_i\}_{i \in \num} \subset \zhat$ be a sequence such that $\lim_{i\to \infty}s_i = s \in \zhat$. Let us denote the metric on $X$ by $d$ and fix an arbitrary $\varepsilon > 0$. We take sequences of positive integers $\{n_{ij}\}_{i,j \in \num}$ such that for every $i$, $n_{ij} \to s_i$ in $\nhat$ as $j \to \infty$. By (ii), we have $d(f(n_{ij}),f(s_i)) < \frac{\varepsilon}{2} \  (j \gg 0)$. Thus we can take a sequence of positive integers $\{j_i\}_{i \in \num}$ such that  $d(f(n_{ij_i}),f(s_i)) < \frac{\varepsilon}{2}$, $n_{ij_i} > i$ and $n_{ij_i} \equiv s_i \mod i!$. Here $n_{ij_i} \to +\infty$ in $\mathbb{R}$ and $n_{ij_i} \to s$ in $\zhat$ as $i \to \infty$, again by (ii) we have $d(f(n_{ij_i}),f(s)) < \frac{\varepsilon}{2} \ (i \gg0)$, thus $d(f(s_i),f(s))< \varepsilon\ (i \gg 0)$.
\end{proof}
\begin{cor}\label{profper}
  Let $X$ be a complete metric space, $f:X \to X$ be a continuous map and $x$ be a point in $X$. Then the following conditions are equivalent:
\begin{enumerate}
\item The map $f\tow_x:\num \to X$ can be extended to a continuous map $\widehat{f \tow_x}:\nhat \to X$.
\item For every $s \in \zhat$ and every sequence of positive integers $\{n_i \}_{i\in \num}$ that satisfies $n_i \to + \infty$ in $\mathbb{R}$ and $\widehat{n_i} \to s$ in $\zhat$ as $i \to \infty$, the limit $ \lim_{i \to \infty} f^{n_i}(x)$ exists in $X$ and is independent of $n_i$ (i.e. limit $ \lim_{i \to \infty} f^{n_i}(x)$ depends only on $s$).
\end{enumerate}
\end{cor}
\begin{deff}
Let $X$ be a complete metric space and $f:X \to X$ be a continuous map.
A point $x$ in $X$ is said to be a {\em profinite preperiodic point} of $f$ if $x$ suffices the conditions (i) or (ii) of Corollary \ref{profper}.
A continuous map $f$ is said to be {\em profinite preperiodic} if every point $x$ in $X$ is a profinite preperiodic point of $f$.
\end{deff}
 We chose the term "profinite preperiodic" since $x$ is a preperiodic point of $f$ if and only if $f \tow_x$ factors through some finite semigroup $\zmn$.

 For a profinite preperiodic point $x$, we define $f^s(x)$ for $s \in \zhat$ by $f^s(x) := \lim_{i\to \infty} f^{n_i} (x)$, where the integer sequence $n_i$ is taken as in Proposition \ref{profper}(ii). If $f$ is profinite preperiodic, then for every $a,b \in \nhat$, we have $f^a\circ f^b = f^{a+b}$ and $(f^a)^b = f^{ab}$.

\section{Dynamical System on $\zhat$}
Before discussing about dynamical systems on $\zhat$, we begin with reviewing basic facts on finite dynamical systems. Let $\alpha:\num \to \num$ be a function of ``the largest factor prime power'', that is, $\alpha(n) = \max \{p^a \mid  p {\ \rm is \ prime \ and\ } p^a \ {\rm divides}\  n\}$. By prime factorization, we can see $\alpha( \lcm \{ n_i \} ) = \max \{ \alpha(n_i) \} $ and $\alpha(mn)\leq \alpha(m)\alpha(n)$.
\begin{lemm}\label{finitemap}
 Let $n$ be a positive integer, $F$ a finite set of cardinality $n$ and $\sigma:F\to F$ a map.
\begin{enumerate}
\item For every $a\in F$, there exist nonnegative integers $k$ and $l$ such that $l\geq1$, $\sigma^{k}(a)=\sigma^{k+l}(a)$ and $k+l \leq n$.
\item For every $a\in F$, if $\sigma^{k_i}(a) = \sigma^{k_i+l_i}(a)$ for $i = 1,2$, then $\sigma^k (a) = \sigma^{k+l}(a)$  holds for $(k,l) = (\min \{k_1, k_2 \}, \gcd \{ l_1, l_2 \})$.
\item For every $a\in F$, there exists an unique pair of nonnegative integers $(k,l)$ such that $\sigma^{k'} (a) = \sigma^{k'+l'}(a)$ if and only if $k \leq k'$ and $l \mid l'$. We denote this pair by $(k_a,l_a)$.
\item Let $K=\max_{a\in F} \{k_a\}$ and $L=\lcm_{a\in F}\{l_a\}$. Then $\sigma^{K'} = \sigma^{K'+L'}$ if and only if $K \leq K'$ and $L \mid L'$.
\item We have $K + \alpha(L) \leq n$.
\end{enumerate}
\end{lemm}
\begin{proof}
(i) Obvious from the pigeonhole principle.

(ii) Without loss of generality, we assume $k = k_1$. We can take a pair of integers $(x,y)$ as $xl_1 + yl_2 = l$. Then we have 
$$ \sigma^{k' +l}(a) = \sigma^{k' +xl_1 + yl_2}(a) = \sigma^{k'}(a)\ (k' \gg 0), $$
particularly for $k' = k+Nl_1$,
$$\sigma^{k + Nl_1}(a) = \sigma^{k+Nl_1 + l}(a) = \sigma^l(\sigma^{k+Nl_1}(a))\ (N \gg 0).$$
 The assumption $k = k_1$ implies $\sigma^{k+Nl_1}(a) = \sigma^k(a)$ for every $N$. Thus $\sigma^k(a)=\sigma^l(f^k(a))=\sigma^{k+l}(a)$.

(iii) Let $k :=\min \{ k' \mid$  $\sigma^{k'}(a) =\sigma^{k'+l'}(a)$ for some $l' \}$ and $l :=\gcd \{ l' \mid$  $\sigma^{k'}(a) =\sigma^{k'+l'}(a)$ for some $k' \}$. By (i) and (ii), the pair $(k,l)$ satisfies the condition.

(iv) Obvious from (iii).

(v) We have
$$K + \alpha(L) = \max \  k_a + \alpha(\lcm \ l_a) = \max k_a + \max \alpha(l_a) \leq \max k_a + \max l_a.$$
 Let $b,c \in F$ be such that $\max k_a = k_b$ and $\max l_a = l_c$. Let 
\begin{eqnarray}
T_b &:= &\{b,\sigma(b),...,\sigma^{k_b-1}(b)\} \nonumber \\
& = &\{\sigma^k(b) \mid  k \in \num {\rm \ such \ that\ } \sigma^k(b) \neq \sigma^{k+l}(b) {\rm \ for\ every\ }l \in \num \}, \nonumber \\
C_c & := &\{\sigma^{k_c}(c),\sigma^{k_c+1}(c),...,\sigma^{k_c+l_c-1}(c)\} \nonumber \\
& = &\{\sigma^k(c) \mid  k \in \num {\rm \ such \ that\ } \sigma^k(c) = \sigma^{k+l_c}(c)\}. \nonumber
\end{eqnarray}
Then $T_b$ and $C_c$ are disjoint subsets of $F$, which leads to $k_b + l_c \leq n$.
\end{proof}

Integers $k=k_a$ and $l=l_a$ are called the {\em tail length} and the {\em cycle length} of $\sigma$ on $a$ respectively. Also, $K$ and $L$ are called the {\em preperiod length} and the {\em period} of $\sigma$ respectively.

\begin{deff}\label{periodmap}Let $f: \zhat \to \zhat$ be a continuous map. A map $P:\num \to \num$ is said to be a \emph{period map} of $f$ that satisfies following equivalent conditions:
\begin{enumerate}
\item For every $s,t \in \zhat$ and $n \in \num$, if $s \equiv_{P(n)} t$ then $f(s) \equiv_n f(t)$.
\item For every $n \in \num$, there exists a map $f_n : \zring_{P(n)} \to \zn$ that makes the following diagram commutative:
\begin{equation}
\begin{CD}
\zhat @>f  >> \zhat \\
@VVV @VVV\\
\zring_{P(n)} @>f_n >> \zn.
\end{CD} \nonumber
\end{equation}
\end{enumerate}
Here the vertical arrows are projections.
\end{deff}
Maps $f_n$ are called {\em reductions} of $f$.

Note that $\zhat$ is metrizable with some metric which is invariant by translations of the form $s \mapsto s + c$. Since $\zhat$ is compact, any continuous map $\zhat \to \zhat$ is also uniformly continuous with such a metric, hence there exists a period map of the continuous map.
 A period map of $f$ is not uniquely determined by $f$. 
\begin{prop}\label{compper}
Let $P$ be a period map of $f$ and $Q$ a period map of $g$. Then $Q \circ P$ is a period map of $f \circ g$. In particular, $P^n$ is a period map of $f^n$.
\end{prop}
\begin{proof}
We check (i) of the definition above. If we take $s,t \in \zhat$ such that $s\equiv_{Q (P(n))} t$, then it follows $g(s) \equiv_{P(n)} g(t)$ and $f(g (s)) \equiv_n f(g(t))$.
\end{proof}

\begin{deff}\label{congpro} Let $f:\zhat \to \zhat$ be a continuous map.
\begin{enumerate}
\item $f$ is said to be $congruence$ $stable$ if there exists a period map $P$ of $f$ such that $P^k(n)=P^{k+1}(n) (k \gg 0)$ for each positive integer $n$.
\item $f$ is said to be $congruence$ $preserving$ if ${\rm id}_\num$, the identity function on $\num$, is a period map of $f$.
\end{enumerate}
\end{deff}
\begin{rem}
Every polynomial in $\zhat[x]$, as a map $\zhat \to \zhat$, is congruence preserving, but not every congruence preserving map is a polynomial with $\zhat$ coefficients, see \cite{Ch} and \cite{Ce-Gr-Gu}.
\end{rem}
\begin{lemm}\label{metaperiod}
Let $f:\zhat \to \zhat$ be a continuous map, $P$ a period map of $f$ and $m=P^k(n)=P^{k+1}(n)$ for $k,m,n \in \num$.
Let $K$ be the preperiod length of the reduction $f_m : \zm \to \zm$ and $L$ be the period of $f_m$.

For every $s,t \in \zhat$ and $u,v \in \num$, if $s \equiv_m t$, $u,v \geq k+K$ and $u \equiv_{L} v$, then we have $f \tow_s(u) \equiv_n f \tow_t(v)$.
\end{lemm}
\begin{proof}If we have $u-k , v-k \geq K$ and $u-k \equiv_{L} v-k$, then 
$$ f^{u-k}(s) \equiv_m f^{v-k}(s) \equiv_m f^{v-k}(t).$$
 By Proposition \ref{compper}, $P^k$ is a period map of $f^k$, we obtain 
$$ f^k(f^{u-k}(s))\equiv_n f^k(f^{v-k}(t)),$$
that is $f \tow_s (u) \equiv_n f \tow_t(v)$.
\end{proof}
\begin{thm}\label{cppp}
Let $f:\zhat \to \zhat$ be a congruence stable map. Then $f$ is profinite preperiodic.
\end{thm}
\begin{proof}
We check the condition (ii) in Proposition \ref{profper}. Let $\{n_i \}$ be a sequence of positive integers such that $n_i  \to \infty$ in $\mathbb{R}$ and $\widehat{n_i} \to s$ in $\zhat$ as $i \to \infty$ for some $s \in \zhat$.

 First, we will show the sequence $\{ f^{n_i}(s) = f \tow_s (n_i) \}_i$ converges, that is, becomes eventually stable modulo $n$ for an arbitrary positive integer $n$. Since $f$ is congruence stable, there is a period map $P$ of $f$ such that $P^k(n) = P^{k+1}(n)$ for some $k$. Therefore we take $K$ and $L$ as in Lemma \ref{metaperiod}. By the definition of the sequence $\{n_i\}$, we have
$$n_i,n_j \geq k+K{\rm\ and\ }n_i \equiv_{L} n_j \ (i,j \gg 0).$$
This leads to 
$$f \tow_x(n_i) \equiv_n f \tow_x(n_j)\ (i,j \gg 0)$$
by the above lemma. Thus $\{ f \tow_x (n_i) \mod n \}_i$ become eventually stable.

 Next, let $\{n'_i\}$ be another sequence that satisfies the same condition as $\{n_i\}$ and converges to the same limit $s$ as $\{n_i\}$ does. Then by a similar argument, we have
$$\lim_{i \to \infty} f\tow_s(n_i) \equiv_n \lim_{i \to \infty} f\tow_s(n'_i)$$
 for every positive integer $n$. This means $\lim_{i \to \infty} f\tow_x(n_i) = \lim_{i \to \infty} f\tow_x(n'_i)$ in $\zhat$.
\end{proof}
 By Proposition \ref{profper}, we obtain a continuous map $\widehat{f \tow_s}: \nhat \to \zhat$. By Proposition \ref{constnhat}, we regard the domain $\nhat$ as $\num \cup \zhat$ and denote the restriction $\widehat{f \tow_s} |_{\zhat} : \zhat \to \zhat$ by $f\ttow_s$, which is also continuous.
\begin{prop} \label{towper}
 Let $f: \zhat \to \zhat$ be congruence stable, $P$ a period map of $f$ and $\mu : \num \to \num$ a map, such that for every $n \in \num$ it holds $\mu(n)=P^k(n)=P^{k+1}(n)$ for some $k \in \num$. Let $\p:\num \to \num$ be a map such that
$$\p(n) = \text{ the period of the reduction } f_{\mu(n)} : \zring_{\mu(n)} \to \zring_{\mu(n)}.$$
Then $\p$ is a period map of $f \ttow_x$ for every $x$ in $\zhat$.

 Moreover, for any $x,y,s \in \zhat$, if $x \equiv_{\mu(n)} y$ then $f\ttow_x(s) \equiv_n f\ttow_y(s)$.
\end{prop}
\begin{proof} We proceed arguement of the proof of Theorem \ref{cppp}. We take $s,t \in \zhat$ as $t \equiv_{L} s$. For any sequence $\{n''_i\}$ such that $n''_i \to t$, we also have $f \tow_x(n_i) \equiv_n f\tow_x (n''_j) (i,j \gg 0)$ by a similar argument to Theorem \ref{cppp}. Therefore we have $f \ttow_x(t) \equiv_n f\ttow_x(s)$.

We shall show the latter part. Let $n$ be a positive integer, $n_i$ a sequence of positive integers such that $n_i \to s$ in $\nhat$. By the assumption $P(\mu(n))=\mu(n)$,  if $n_i \geq k$ then we have
$$ f^{n_i-k}(x) \equiv_{\mu(n)} f^{n_i-k}(y)$$
then by the assumption $P^k(n)=\mu(n)$ and Proposition \ref{compper}, we have
$$f \tow_x (n_i) = f^{n_i}(x) \equiv_n f^{n_i}(y) = f\tow_y(n_i) \ (i \gg 0).$$
Passing to the limits, we obtain $f \ttow_x (s) \equiv_n f\ttow_y(s)$.
\end{proof}
 We now discuss the case $f:\zhat \to \zhat$ is congruence preserving. We denote the period of the reduction $f_n :\zn \to \zn$ by $\p_f(n)$, defining $\p_f:\num \to \num$. By the above proposition, we have the following:
\begin{prop}\label{towcongper}
Let $f: \zhat \to \zhat$ be congruence preserving and $s \in \zhat$. Then $\p_f$ is a period map of $f \ttow_s$.
\end{prop}
Now we shall evaluate $\p_f$.
\begin{lemm}\label{p_f<1}
 Let $f:\zhat \to \zhat$ be congruence preserving and $p$ a prime number. The following conditions are equivalent each other:
\begin{enumerate}
\item $\p_f(p) \neq p$.
\item The reduction $f_p:\zp \to \zp$ of $f$ is not a cyclic permutation of length $p$.
\item For every $x \in \zp$, it holds $l_x < p$ for the cycle length $l_x$ of $f_p$ on $x$.
\item $\alpha(\p_f(p)) < p$.
\end{enumerate}
\end{lemm}
\begin{proof}
((i) $\Rightarrow$ (ii)) The contraposition is obvious.

((ii)  $\Rightarrow$ (iii)) We show the contraposition. Assume that there is some $x \in \zp$ such that $f_p^k(x) = {f_p}^{k+l}(x)$ implies $l \geq p$. Then $x,f_p(x),...,(f_p)^{p-1}(x)$ are distinct elements of $\zp$ and $(f_p)^p(x) =x$, that is, $f_p$ is a cyclic permutation of length $p$.

((iii)  $\Rightarrow$ (iv)) Assume (iii). By Lemma \ref{finitemap} (iv), the period $\p_f(p)$ of $f_p:\zp \to \zp$ satisfies
$$\alpha (\p_f(p)) = \alpha( \lcm\{ l_x\} ) = \max \{ \alpha(l_x) \} \leq  \max \{l_x\} < p,$$
where $x$ runs over $\zp$.

((iv) $\Rightarrow$ (i)) The contraposition is obvious, because $p$ is a prime number.
\end{proof}

\begin{deff}\label{ts}
A congruence preserving function $f: \zhat \to \zhat$ is said to be \emph{tower-stable} if $f$ satisfies $\alpha(\lambda_f(p))< p$, the condition (iv) of Lemma \ref{p_f<1}, for every prime number $p$. 
\end{deff}
\begin{lemm} \label{per}

Let $f$ be a congruence preserving map and $n,n'$ be positive integers. The followings hold:
\begin{enumerate}
\item $\lcm(\p_f(n),\p_f(n'))=\p_f(\lcm(n,n'))$. In particular, if $n  \mid  n'$ then $\p_f(n)  \mid  \p_f(n')$.
\item $\p_f(p^{a+1}) \mid  \lcm\{1,2,\ldots, p\}\cdot \p_f(p^a)$ for every prime $p$ and every positive integer $a$.
\item $\p_f^k(n) = \p_f^{k+1}(n) \ (k \gg 0)$.
\end{enumerate}
Moreover, if $f$ is tower-stable, the assertions (iii) can be strengthen as
\begin{itemize}
\item[ (iii')] $\p_f^k(n) = 1 \ (k \gg 0)$.
\end{itemize}
\end{lemm}

\begin{proof}
(i) We recall that for any positive integers $N$ and $\nu$, $f^{M}(x)\equiv_\nu f^{M+N}(x)\ (M \gg 0)$ for all $x \in \zhat$ if and only if $\p_f(\nu) \mid N$ by Lemma \ref{finitemap}(iv). Therefore we have
\begin{eqnarray}
 & & \p_f(\lcm(n,n')) \mid N \nonumber \\
& \Leftrightarrow & f^M(x)\equiv_{\lcm(n,n')} f^{M+N}(x)\ (M \gg 0) {\rm \ for\ all\ }x \in \zhat \nonumber \\
& \Leftrightarrow & f^M(x)\equiv_n f^{M+N}(x) \ {\rm and} \ f^M(x) \equiv_{n'} f^{M+N}(x)\ (M \gg 0)  {\rm \ for\ all\ }x \in \zhat \nonumber \\
& \Leftrightarrow & \p_f(n) \mid N  \ {\rm and} \  \p_f(n') \mid N \nonumber
\end{eqnarray}
for every positive integer $N$. This completes the proof.

(ii) We set $l := \p_f(p^a)$. Let us take an arbitrary $s \in \zhat$. We put $k$ the preperiod length of the reduction $f_{p^a}:\zring_{p^a} \to \zring_{p^a}$ and $t = f^k(s)$. Let $\pi : \zhat \twoheadrightarrow \zring_{p^a} $ and $\pi_a : \zring_{p^{a+1}} \twoheadrightarrow \zring_{p^a}$ be the cannonical surjections. We define
$$T := \pi_a^{-1}(\pi(t)) = \{\ol{t + cp^a} \mid 0 \leq c <p\} \subset \zring_{p^{a+1}}.$$
 For every $\ol{t'} \in T$, from $t = f^k(s)$ it follows that
$$ f^l(t') \equiv_{p^a}  f^l(t) = f^{k+l}(s) \equiv_{p^a} f^k(s) = t ,$$
 that is $f^l (T) \subset T$. Therefore by Lemma \ref{finitemap} (i), we can take $k_t$ and $l_t$ as 
$$ f^{lk_t}(t) \equiv_{p^{a+1}} f^{l(k_t+l_t)}(t) \ \  {\rm and} \ \ k_t+l_t \leq {\rm card}(T) = p,$$
 thus in particular,
$$f^{k+lp}(s) = f^{lp}(t) \equiv_{p^{a+1}} f^{l(p+\lcm[p])}(t) = f^{k+lp + l\cdot \lcm[p]}(s)$$
since $k_t \leq p$ and $l_t \mid \lcm \{1,2,...,p \} = \lcm[p]$.
 Thus Lemma \ref{finitemap} (iv) leads
$$\p_f(p^{a+1}) \mid l\ \lcm[p] = \p_f(p^a)\lcm[p].$$

(iii) We prove this by induction on $\alpha(n)$. The case $\alpha(n) = 1$ is obvious. Let us assume that the assertion holds for every $n$ of $\alpha(n) < p^a$ and now set $\alpha(n) = p^a$. By (ii), we have $\alpha(\p_f(p^a)) \leq p \cdot \alpha(\p_f(p^{a-1}))$ thus
\begin{equation}\label{pa}
\alpha(\p_f(p^a)) \leq p \cdot \alpha(\p_f(p^{a-1})) \leq \cdots \leq p^{a-1} \alpha( \p_f(p)) \leq p^a 
\end{equation}
and hence $\alpha(\p_f(n)) \leq \alpha(n)$ by (i).
The case $\alpha(\p_f(n)) < \alpha(n)$ is evident from the assumption of the induction.
 Let us assume $\alpha(\p_f(n))=\alpha(n)=p^a$, that is $\alpha(\p_f(p^a))=p^a$, then put $\p_f(p^a) = cp^a$. By (i) of this lemma, we have
$$\p_f(n) = \lcm \left\{ \p_f(p^a), \p_f \left(\frac{n}{p^a}\right) \right\}= \lcm \left\{p^a, c ,\p_f \left(\frac{n}{p^a} \right)\right\}.$$
Therefore, by applying (i) inductively, we obtain
$$\p_f^k(n) = \lcm \left\{ p^a, c,\p_f (c) ,...,\p_f^{k-1}(c) , \p_f^k\left(\frac{n}{p^a}\right) \right\}.$$
Here $\alpha(c)$ and $\alpha\left(\frac{n}{p^a}\right)$ are both less than $p^a$, therefore by the assumption of the induction, $\p_f^k(n)$ is eventually constant.

(iii') (ii) and $\alpha(\p_f(p)) < p$ leads to $\alpha(\p_f(p^a)) < p^a$ in (\ref{pa}). This leads to $\alpha(\p_f(N))<\alpha(N)$ for every $N>1$. Therefore, we have $\alpha(\p_f^k(n)) = 1 \ (k \gg 0)$, that is, $\p_f^k(n) = 1 \ (k \gg 0)$.
\end{proof}

\begin{thm}\label{convzhat}
 Let $f$ be congruence preserving and $s \in \zhat$. Then the map $f \ttow_s$ is profinite preperiodic. Moreover, if $f$ is tower-stable, then $f \ttow_s \ttow_t$ is a constant function and independent of $t$.
\end{thm}
\begin{proof}
 By Proposition \ref{towcongper}, $\p_f$ is a period map of $f\ttow_s$. Therefore by Lemma \ref{per} (iii), $f \ttow_s$ is congruence stable. Theorem \ref{cppp} shows that $f \ttow_s$ is profinite preperiodic.

 We now assume that $f$ is tower-stable. By Lemma \ref{per} (iii'), we can apply Proposition \ref{towper} for $f \ttow_s$ and $\p_f$, with the constant functions $\mu \equiv 1$ and $\p \equiv 1$. We therefore obtain that if $t \equiv_1 t'$ and $u \equiv_1 u'$, then
$$f \ttow_s \ttow_t (u) \equiv_n f \ttow_s \ttow_t (u') \equiv_n f \ttow_s \ttow_{t'} (u')$$
for every $n$. Thus $f \ttow_s \ttow_t (u) = f\ttow_s \ttow_{t'}(u')$ for any $t,t',u,u' \in \zhat$.
\end{proof}

\begin{proof}[Proof of Theorem \ref{main1}]
((i)$\Rightarrow$(ii))
Let us denote by $\widehat{\cdot}:\zahl \to \zhat, a \mapsto \widehat{a}$ the natural embedding. For a given polynomial $f(x) = \sum a_i x^i \in \zahl[x]$, set $\widehat{f}(x) := \sum \widehat{a_i}x^i \in \zhat[x]$. Let $F: \nhat \to \nhat$ be a map such that
\[ F(n) = 
	\begin{cases}
		f \tow_a(n) = f^n(a) \in \num \  (n \in \num),\\
		\widehat{f} \ttow_{\widehat{a}} (n) = \widehat{ \widehat{f} \tow_{\widehat{a}}}(n) \in \zhat \ (n \in \zhat).
	\end{cases}
\]
Let $n_i$ be a sequence of positive integers such that $n_i \to \infty$ in $\mathbb{R}$ and $n_i \to s \in \zhat$ in $\zhat$ as $i \to \infty$. Then we have
\[
\widehat{F(n_i)}=\widehat{f^{n_i}(a)}=\widehat{f}^{n_i}(\widehat{a})=\widehat{f}\tow_{\widehat{a}}(n_i) \to \widehat{f}\ttow_{\widehat{a}}(s) = F(s) \ \ {\rm in} \ \zhat \ \  (i \to \infty).
\]
If there exists a sequence $F(n_i) \nrightarrow \infty$ in $\mathbb{R}$, then $F(n_i) = F(n_{i'})$ holds for some distinct nonnegative integers $i$ and $i'$, that is, $a$ is a preperiodic point of $f$. Therefore the function $f \tow_a$ is bounded on $\num$ and (ii) automatically hold.

Therefore we can assume $F(n_i) \rightarrow \infty$ in $\mathbb{R}$ and thus the map $F$ is continuous on $\nhat$.
By Lemma \ref{p_f<1} (ii) $\Rightarrow$ (i) and Theorem \ref{convzhat}, there exists $t \in \zhat$ such that $F^n(s) = (\widehat{f} \ttow_{\widehat{a}} )^n(s) \to t$ as $n \to \infty$ for every $s \in \zhat$. We note that this is equivalent to $\bigcap_{n=1}^{\infty}F^n(\zhat) = \{ t \}$ because $\zhat$ is compact and Hausdorff.

Let $b$ be a positive integer. We assume $F^n(b) \nrightarrow t$, that is, there exists some neighborhood $N$ of $t$ in $\zhat$ and some subsequence $F^{n_1}(b),F^{n_2}(b),...$ of $F^{n}(b)$, such that $n_i$ is increasing and $\{ F^{n_i}(b)\} \cap N = \emptyset$.

Since $\zhat$ is compact metric space, there exists a converging subsequence $F^{n_{k_i}}(b) \to s \in \zhat$. We write $m_i = n_{k_i}$. For every $i$, $F^{m_i}(F^{m_{i+j}-m_i}(b))=F^{m_{i+j}}(b) \to s$ as $j \to \infty$. Again $\zhat$ is compact metric space, let $F^{m_{i+j_l}-m_i}(b) \to s'$ be a converging subsequence of $\{ F^{m_{i+j}-m_i}(b) \}_j$, then we have $s = F^{m_i}(s') \in F^{m_i}(\zhat)$. This holds for any $i$, therefore $s \in \bigcap_{i=1}^{\infty}F^{m_i}(\zhat)=\bigcap_{n=1}^{\infty}F^n(\zhat) = \{t\}$, a contradiction. Thus, for any $b$, $F^n(b) = f\tow_a \tow_b (n)$ converges to the same $t \in \zhat = \prod_p \zahl_p$ and assertion (ii) follows.

((ii)$\Rightarrow$(i)) We don't use profinite complation here.

We prove contraposition. We assume that for some prime number $p$ and $x \in \zahl$,  $f^b(x) \equiv_p f^c(x)$ if and only if $b \equiv_p c$ for any positive integers $b$ and $c$. This leads that for every $a \in \num$, $f\tow_a : \num \to \num$ induces the bijection $(f\tow_a)_p : \zp \to \zp$. We regard $(f \tow_a)_p$ as a permutation on $\zp$. The assumption also leads that $f$ is a non-constant polynomial, therefore we can take $a \in \num$ as that $\lim_{n \to \infty} f \tow_a \tow_b (n) = \infty$ in $\real$ holds for any $b \in \num$.

 If $(f\tow_a)_p$ has two or more cycles, then we take $b,b' \in \num$ as that $\bar{b}$ and $\bar{b'}$ are in distinct cycles of $(f\tow_a)_p$ on $\zp$. In particular, for every $n$, we have $f\tow_a\tow_b(n) = (f\tow_a)^n(b) \not \equiv_{p} f\tow_a\tow_{b'}(n)$. This contradicts to the independence of $b$ in (ii).

 If $(f\tow_a)_p$ is a permutation of one cycle, then the cycle length is $p$ and $f\tow_a\tow_b$ becomes a permutation on $\zp$ again. Therefore the sequence $\{ f\tow_a\tow_b(n) \mod p \}_n$ cannot be eventually stable.
\end{proof}

\section{Period Map of Iterated Polynomial}
 Throughout this section, we assume $f(x) = \sum_i c_i x^i \in \zahl[x]$ is a tower-stable polynomial. For every integer $a$, let $\od_{f,a}(n)$ (resp. $\p_{f,a}(n)$) be the tail (resp. cycle) length on $\ol{a} \in \zn$ of the reduction $f_n:\zn \to \zn$ of $f$.
 To show Theorem \ref{main2}, we evaluate $\p_{f,a}$ and $\od_{f,a}$. Evaluation of $\p_{f,a}$ is established by Fun and Liao in \cite{F}, in a context of $p$-adic dynamics. But one for $\od_{f,a}$ is not explicitly written, although implicitly in evaluating process of $\p_{f,a}$, because it has not been of principal interest. Therefore we reevaluate the tail length. We use following Lemmas 4.1 and 4.2.
\begin{lemm}
Let $a,c$ and $n$ be positive integers and $p$ prime number.  The following holds:
\begin{enumerate}
\item $f(a+cp^n) \equiv_{p^{n+1}} f(a) + cf'(a)p^n $,
\item (chain rule) $(f^n)' = \prod_{i=0}^{n-1} f'\circ f^i$
\end{enumerate}
where $g'$ denotes the derivative of $g$ for each polynomial $g$.
\end{lemm}
\begin{lemm}
 Let $p$ be a prime number, $g(x)=bx+c \in \zahl[x]$ be an linear polynomial and $s$ an integer. Then we have 
$$\p_{g,s}(p) \mid
\begin{cases}
1 & {\rm if} \  b \equiv_p 0, \\
p & {\rm if} \  b \equiv_p 1 
\end{cases} 
\ {\rm and}\ 
\od_{g,s}(p) \leq
\begin{cases}
1 & {\rm if} \ b \equiv_p 0, \\
0 & {\rm if} \ b \equiv_p 1.
\end{cases}
$$
\end{lemm}
\begin{proof} See [F-L, Lemma 1].
\end{proof}
 From now, we fix a polynomial $f$ and a (positive) integer $a$. We denote $\od_{f,a}$ by $\od$, and $\p_{f,a}$ by $\p$.
 We define {\em mod $p$ multiplier} of $f$ on $a$ by
$$\mu_p := (f^{\p(p)})'(f^{\od(p)}(a)) = \prod_{i=0}^{\p(p)-1}f'(f^{\od(p)+i}(a)).$$
\begin{thm}\label{polyper}
Let $k$ be a positive integer and $p$ a prime number. 
\begin{enumerate}
 \item If $\mu_p \equiv_p 0$, then it holds $\p(p^k) = \p(p)$ and $\od(p^k) \leq \od(p) + (k-1) \cdot \p(p)$.
 \item If $\mu_p \not \equiv_p 0$, then it holds $\od(p^k)=\od(p)$ and $\p(p^k) \mid  \p(p) \cdot (p-1) \cdot  p^{k-1}$.
\end{enumerate}
\end{thm}
\begin{proof}
 From Lemma \ref{finitemap} and an evident inequality $\od(p^k) \geq \od(p)$, it is enough to show $f^{K+L}(a) \equiv_{p^k} f^K(a)$ for 
\begin{equation}
(K,L) = 
\left\{
\begin{array}{llll}
(\od(p)+(k-1)\cdot \p(p), & \p(p) & ) & {\rm if}\ \mu_p \equiv_p 0, \\
(\od(p), & \p(p) \cdot (p-1) \cdot p^{k-1} & )& {\rm if}\ \mu_p \not \equiv_p 0.
\end{array} \label{KLdef}
\right.
\end{equation}

We show this by induction on $k$. The case $k=1$ is evident.
 We take a pair $(K,L)$ as above and assume that they satisfy
$$f^{K+L}(a) \equiv_{p^k} f^K(a).$$
We put $t := f^K(a)$ and get $$t \equiv_{p^k} f^{L}(t) =: cp^k + t \ (c \in \zahl).$$ 
 Applying Lemma 4.1(i) to $f^L$, we obtain
$$f^L(t + sp^k) \equiv_{p^{k+1}} s(f^L)'(t) p^k+ cp^k + t = (bs+c)p^k + t\text{ with } b:= (f^L)'(t),$$ that is $$f^L(t+sp^k) \equiv_{p^{k+1}} t + g(s)p^k,$$ where $g(s) := bs+c$. This implies
$$(f^L)^n(t+sp^k) \equiv_{p^{k+1}} t + g^n(s)p^k.$$
Therefore if $g^m(s) \equiv_p g^{m + n}(s)$ for some $m$ and $n$, we have
\begin{equation}
(f^L)^{m}(t) \equiv_{p^{k+1}} (f^L)^{m + n}(t). \label{KL}
\end{equation}

Applying Lemma 4.1(ii) to $f^L$, we have 
\begin{eqnarray}
b  = (f^L)'(t) = \prod_{i = 0}^{L-1} f'(f^i(t)) = \prod_{i=0}^{\frac{L}{\p(p)}-1} \prod_{j=0}^{\p(p)-1} f'(f^{i\p(p)+j}(t)) \nonumber \\
 \equiv_p  \prod_{i=0}^{\frac{L}{\p(p)}-1}\prod_{j=0}^{\p(p)-1} f'(f^{ \od(p) + j}(a)) = \mu_p^{\frac{L}{\p(p)}} \nonumber
\\
=
\begin{cases}
\mu_p & \equiv_p  0 \ \ \ {\rm if} \ \mu_p \equiv_p 0, \\
\mu_p^{(p-1) \cdot p^{k-1}} & \equiv_p 1\ \ \ {\rm if} \ \mu_p \not \equiv_p 0
\end{cases}
 \label{bmu}
\end{eqnarray}
 where the first congruence follows from 
$$f^{i\p(p)}(t) = f^{i\p(p)+K}(a) = f^{i'\p(p)+\od(p)}(a) \equiv_p f^{\od(p)}(a)$$
and $f^j,f' \in \zahl[x]$ (this implies they are congruence preserving), the last congruence in (\ref{bmu}) for the case $\mu_p \not \equiv_p 0$ follows from Fermat's little theorem.
 Applying Lemma 4.2 to $g(s) = bs+c$, by (\ref{bmu}) we obtain $g^m(s) \equiv_p g^{m+n}(s)$ for
$$
(m,n) = 
\begin{cases}
(1,1) & {\rm if} \ \mu_p \equiv_p 0, \\
(0,p) & {\rm if} \ \mu_p \not \equiv_p 0.
\end{cases} 
$$
 By subsutituting the above pair $(m,n)$, $t = f^K(a)$ and (\ref{KLdef}) into (\ref{KL}), we complete the induction.
\end{proof}
\begin{prop}
 If $\mu_p \equiv_p 0$, then the $p$-adic part of $f \ttow_a (s)$ is algebraic number for every $s \in \zhat$.
\end{prop}
\begin{proof}
Let $s'$ be a positive integer such that $s'>\od(p)$ and $s' \equiv_{\p(p)} s$.
By the above theorem, for every $k \in \num$ we have
$$f^{k\p(p)+s'}(a) \equiv_{p^k} f^{(k+1)\p(p)+s'}(a).$$
Therefore the sequence $\{x_k:=f^{k\p(p)+s'}(a)\}$ converges in $\zahl_p$  as $k \to \infty$. Since $x_{k+1} = f^{\p(p)}(x_k)$, $x := \lim_{k\to \infty}x_k$ satisfies $f^{\p(p)}(x)=x$, hence $x$ is algebraic over $\mathbb{Q}$.

 For every increasing sequence of positive integers $\{n_i\}$ such that $n_i \to s$ in $\nhat$ as $i \to \infty$, $\{f^{n_i}(a)\}$ is a partial sequence of $\{x_k\}$ eventually. Thus $f \ttow_a(s) = \lim_{i \to \infty} f^{n_i}(a) = x$ is algebraic.
\end{proof}

\begin{prop}
Let $b$ be an $f$-valid number. Then it holds $\p_f(b^{n+1}) \mid b^n$ for every sufficiently large $n$.
\end{prop}
\begin{proof}
From Lemma \ref{per}, Theorem \ref{polyper} and $b$ is square-free,
$$\p_f(b^{n+1}) = \mathop{\lcm_{p \mid b}} \{\p_f(p^{n+1})\}  \mid  \mathop{\lcm_{p \mid b}} \{\p_f(p) \cdot (p-1) \cdot p^n\}.$$
For $p$ with $p \mid b$, if we decompose $\p_f(p)\cdot(p-1)=\prod q_i^{r_i}$, then since $b$ is $f$-valid, we have $q_i \mid b$ for every $q_i$. Therefore if we put  $N = \max_p \max_{q_i} r_i$, then we have $\p_f(p) \cdot(p-1) \cdot p^n  \mid  b^n$ for every prime $p \mid b$ and integer $n \geq N$.
\end{proof}

\begin{lemm}Let $\od$ and $\p$ be positive integers, $p$ prime, $x \in \zahl_p$ be $p$-adic integer that is algebraic over $\mathbb{Q}$ and $\{x_n\}$ be an nondecreasing sequence of positive integers such that $x_n \equiv_{p^n} x$ as $p$-adic integers. If $x_n$ is not eventually constant, then we have 
$$x_n \geq \p n + \od$$
for every sufficiently large $n$.
\end{lemm}
\begin{proof}
Let $F$ be a polynomial with integer coefficients such that $F(x)=0$. Then $x_n \equiv_{p^n} x$ leads to
$$F(x_n) \equiv_{p^n} F(x) = 0.$$
We have $F(x_n) \neq 0 (n \gg 0)$, because $x_n$ is nondecreasing and not eventually constant. Hence $$|F(x_n)| \geq p^n,$$ this leads $x_n \geq O(p^{n/\deg F}) > O(n)$ as $n \to \infty$.
\end{proof}

\begin{proof}[Proof of Theorem \ref{main2}]
By Theorem \ref{convzhat}, we can take the limit $x = \lim_{k \to \infty}f\ttow_a \tow_t (k)$.
 This $x$ satisfies $f \ttow_a (x) = x$. Write the expansion of the $b$-adic part of $x$ as
$$x = c_0 +c_1b +c_2b^2+...+c_nb^n+... \; ,\ {\rm where} \ 0 \leq c_n < b$$ and $x_{n|}$ the positive integers given by the first $n$ digits of the expansion, that is,
$$x_{n|} := c_0 + c_1b+c_2b^2+...+c_{n-1}b^{n-1}.$$
Then the sequence $x_{n|}$ is nondecreasing.
 First, we assume the sequence $x_{n|}$ is not eventually constant. We show that 
\begin{equation}
  \od(b^n) \leq x_{n|} \ (n \gg 0). \label{xnodbn}
\end{equation}
 We have $\od(b^n) = \max_{p^a \mid b} \{ \od(p^{an}) \}$ by a similar arguement as Lemma \ref{per}(i). Therefore by Theorem 4.3, we have 
\begin{equation}
\od(b^n) \leq l n + k \ (n \gg 0) \label{bnlnk}
\end{equation}
for some integers $l,k$. Moreover, if $\mu_p \not \equiv_p 0$ for every $p \mid b$, (\ref{bnlnk}) holds for $l = 0$ and (\ref{xnodbn}) follows obviously. We assume $\mu_p \equiv_p 0$ for some $p$. By Proposition 4.4, the $p$-adic component of $x$ is algebraic, and $x_{n|} \equiv_{p^n} x$ by the definition of $x_{n|}$. Thus by Lemma 4.6, $x_{n|} > l n+ k \ (n \gg 0)$. Combining this with (\ref{bnlnk}) gives (\ref{xnodbn}).
 
By Proposition 4.5, $x_{n|} \equiv_{b^n} x$ implies 
\begin{equation}\label{xnbnxn}
x_{n|} \equiv_{\p(b^n)} x \ (n \gg 0).
\end{equation}
From (\ref{xnodbn}) and (\ref{xnbnxn}),
\begin{eqnarray}
f^{x_{n|}}(a) &\equiv_{b^n} &f^x(a) \ (n \gg 0) \nonumber \\
& = & x \equiv_{b^n} x_{n|}. \nonumber
\end{eqnarray}
Therefore taking $x_n$ as
$$
x_n := 
\begin{cases}
x_{N|}\ {\rm if}\ n \leq N,\\
x_{n|}\ {\rm otherwise},
\end{cases}
\ (N \gg 0)
$$
satisfies 
\begin{equation}
f^{x_n}(a) \equiv_{b^n} x_n \ \ {\rm and}\ \ x_n = c'_n b^{n-1} + x_{n-1}\ \ \ (0 \leq c'_n < b), \nonumber
\end{equation}
the required condition.

If $x_{n|}$ is eventually constant, then let us denote the constant by $x'$ (to distinguish from $x \in \zhat$). We replace $x_{n|}$ as $x_{n|} := x' + b^n$ and apply the above argument. (In this case, the argument is easier because (\ref{bnlnk}) and the defitnition of $x_{n|}$ induces (\ref{xnodbn}) directly.)
\end{proof}

\section{Remaining Problems}
All remaining problems of tetration raised in \cite{S-S} is generalizable to iterated polynomial. The most interesting one is for the limit $$^\infty f(a) := \lim_{n \to \infty} f\tow_a \tow_b (n) \in \zhat.$$
\begin{prob}
Are $^\infty f(a)$ irrational, and transcendental over $\mathbb{Q}$ except trivial counterexamples (such that $^\infty f(a) \in \num$, or $\mu_p \equiv_p 0$: see Proposition 4.4)? Are there some nontrivial algebraic, or analytic correlations over the set $\{^\infty f(a)\}_{f,a}$?
\end{prob}
Relatively many of polynomials are tower-stable.
Theorem \ref{convzhat} and Lemma \ref{p_f<1} (ii) lead a probability of tower-stable function $C_{tow}$ as 
\begin{eqnarray*} C_{tow} & := &\int_{\fct{Cong}(\zhat)} \{f \mid f \text{ is tower-stable} \}\ d\mu \\
& = & \prod_p \int_{\fct{Cong}(\zp)} \{ f \mid f \text{ is not a cycle permutation of length } p \}\ d\mu_p \\
& = & \prod_p \left(1-\frac{(p-1)!}{p^p} \right) \sim 0.688,
\end{eqnarray*}
where we denote by $\fct{Cong}(X)$ the set of all congruence preserving functions, or the set of their reductions, on $X$. $\mu$ and $\mu_p$s are the additive Haar probability measures.
\begin{prob}
Is $C_{tow}$ a transcendental number?
\end{prob}
In Chapter 4, we employed polynomial $f$. However, some congruence preserving maps on $\num$ are not polynomial.  For instance, the map $n \mapsto \lceil e^{n}n! \rceil$ is introduced in \cite{Ce-Gr-Gu}.
\begin{prob}
Does Theorem \ref{main2} hold for every congruence preserving map?
\end{prob}


\begin{thebibliography}{99}
\bibitem[Ce]{Ce}P. C\'{e}gielski, Arithmetical Congruence Preservation: From Finite to Infinite, Fields of Logic and Computation II 210-225, arXiv:1506.00149.

\bibitem[Ce-Gr-Gu]{Ce-Gr-Gu}P. C\'{e}gielski, S. Grigorieff, I. Guessarian, Newton representation of functions
over natural integers having integral difference ratios, International Journal of
Number Theory 11(7): 2109-2139 (2015) arXiv:1310.1507

\bibitem[Ch]{Ch}Zhibo Chen, On polynomial functions from $\zahl_n$ to $\zahl_m$, Discrete Mathematics 137 (1995) 137-145

\bibitem[F-L]{F}A. Fan and L. Liao, On minimal decomposition of p-adic polynomial dynamical systems, Advances in Mathematics 228, Issue 4 (2011) arXiv:1010.5583

\bibitem[J-Y]{J-Y}J. Jim\'{e}nez Urroz and J. Luis A. Yebra, On the Equation $a^x \equiv x (\mod b^n)$, Journal of Integer Sequences 12 (2009)

\bibitem[K]{K}Donald E. Knuth, Mathematics and Computer Science: Coping with Finiteness, Science 194 (1976)

\bibitem[S-S]{S-S}Daniel B. Shapiro, S. David Shapiro, Iterated Exponents in Number Theory, Integers: Electronic Journal of Combinatorial Number Theory 7 (2007)

\end{thebibliography}
\end{document}